\documentclass[final]{siamltex}

\usepackage{amsmath}
\usepackage{amsfonts}
\usepackage{eucal}

\newtheorem{assumption}[theorem]{Assumption}

\newcommand{\Z}{\mathbb{Z}}                     
\newcommand{\R}{\mathbb{R}}                     
\newcommand{\sph}{\mathcal{S}}                  
\newcommand{\hilb}{\mathcal{H}}                 
\newcommand{\ip}[2]{\langle#1,#2\rangle}        
\newcommand{\abs}[1]{\lvert#1\rvert}            
\newcommand{\norm}[1]{\lVert#1\rVert}           
\newcommand{\cdotc}{\,\cdot\,}                  

\title{Extending the range of error estimates for radial
approximation in Euclidean space and on spheres}

\author{R. A. Brownlee\footnotemark[2]\ \footnotemark[3] \and E. H. Georgoulis\footnotemark[2]
\and J. Levesley\footnotemark[2]}

\begin{document}

\maketitle
\renewcommand{\thefootnote}{\fnsymbol{footnote}}
\footnotetext[2]{Department of Mathematics, University of Leicester,
Leicester LE1 7RH, England}
\footnotetext[3]{This author was
partially supported by a studentship from the Engineering and
Physical Sciences Research Council.}
\renewcommand{\thefootnote}{\arabic{footnote}}

\begin{abstract}
We adapt Schaback's error doubling trick~[R.~Schaback.
\newblock Improved error bounds for scattered data interpolation by radial
  basis functions.
\newblock {\em Math. Comp.}, 68(225):201--216, 1999.]
to give error estimates for radial interpolation of functions with
smoothness lying (in some sense) between that of the usual native
space and the subspace with double the smoothness. We do this for
both bounded subsets of $\R^d$ and spheres. As a step on the way to
our ultimate goal we also show convergence of pseudoderivatives of
the interpolation error.
\end{abstract}

\begin{keywords}
multivariate interpolation, radial basis functions, error estimates,
smooth functions
\end{keywords}

\begin{AMS}
41A05, 41A15, 41A25, 41A30, 41A63
\end{AMS}

\pagestyle{myheadings} \thispagestyle{plain} \markboth{R. A.
BROWNLEE, E. H. GEORGOULIS, AND J. LEVESLEY}{EXTENDING THE RANGE OF
ERROR ESTIMATES}

\section{Introduction}

In this paper we are interested in extending the range of
applicability of error estimates for radial basis function
interpolation in Euclidean space and on spheres. Let $\Omega$ be a
subset of $\R^d$, or the sphere.  Let $d(x,y)$ denote the distance
between two points in $\Omega$. Let $Y \subset \Omega$ be a finite
set of points, and measure the \textit{fill-distance} of $Y$ in
$\Omega$ with
\begin{equation*}
    h(Y,\Omega) := \sup_{x \in \overline{\Omega}} \min_{y \in Y} d(x,y).
\end{equation*}

Given a univariate function $\phi$ defined either on $\R_+$ or
$[0,\pi]$, depending on whether we are in Euclidean space or on the
sphere, we form an approximation
\begin{equation*}
    S^Y_\phi (x) = \sum_{y \in Y} \alpha_y \phi(d(x,y)).
\end{equation*}
If the coefficients $\alpha_y$, $y \in Y$, are determined by the
interpolation conditions
\begin{equation*}
    S_\phi^Y(y) = f(y),\qquad \text{for $y \in Y$,}
\end{equation*}
we refer to $S^Y_\phi$ as the $\phi$\nobreakdash-spline interpolant to $f$ on $Y$.

We will be approximating functions $f \in \hilb_\phi$, a Hilbert
space of functions which depends on the function $\phi$---the
so-called \textit{native space}. Later we will be more explicit
about this space of functions. With this Hilbert space we have a
inner product $\ip{\cdotc}{\cdotc}_\phi$, with associated norm
$\norm{\cdotc}_\phi$. We will require the following useful
orthogonality and consequent \textit{Pythagorean} property; see,
e.g.,~\cite{bgl_levesley03,bgl_light98}.
\begin{proposition}\label{bgl_normmin}
Let $S_\phi^Y$ be the $\phi$\nobreakdash-spline interpolant to $f$
on the point set $Y \subset \Omega$. Then, for all $f \in
\hilb_\phi$,
\begin{enumerate}
    \item $\ip{f - S_\phi^Y}{S_\phi^Y}_\phi = 0$;
    \item $\norm{f}_\phi^2 = \norm{f-S_\phi^Y}_\phi^2+\norm{S_\phi^Y}_\phi^2$.
\end{enumerate}
\end{proposition}

The usual error estimate for $\phi$\nobreakdash-spline interpolants
is of the form
\begin{equation*}
    \abs{f(x) - S_\phi^Y(x)} \leq P(x,Y,\phi) \norm{f - S_\phi^Y}_\phi,
\end{equation*}
where estimation of $P(x,Y,\phi)$---the so-called \textit{power
function}---leads to error estimates for interpolation in terms of
the fill-distance $h(Y,\Omega)$. For the archetypal function in
$\hilb_\phi$ we can say no more than $\norm{f - S_\phi^Y}_\phi
\rightarrow 0$ as $h(Y,\Omega) \rightarrow 0$. However, if $f$ has
double the smoothness (in some sense to made clear later) than the
typical function, then Schaback~\cite{bgl_schaback99} has shown how
to double the convergence order of the $\phi$\nobreakdash-spline
interpolant.

We show how to get improved orders of convergence when the target
function, $f$, has less smoothness than Schaback requires, but more
smoothness than the typical function. We shall be doing this on the
sphere (though this can be easily generalised to other two-point
homogeneous spaces) in \S\ref{bgl_sphere} and in Sobolev spaces on
Euclidean space in \S\ref{bgl_euclid}. An intermediate result in
both cases is to prove approximation orders for pseudoderivatives of
the interpolant. We will define this notation at the appropriate
place in each of the following sections.

In each case we shall be concerned with the practical scenario in
which $Y$ consists of a finite number of points. Foregoing this
assumption is of theoretical interest. In particular, in Euclidean
space for (perturbed) gridded data, certain improved error estimates
are already known to hold for functions within the native space
itself (see, e.g.,~\cite{bgl_buhmann03}).

The goal in this paper is quite different from the desire to
establish error estimates for functions possessing insufficient
smoothness for admittance in the native space. In recent years,
contributions in that direction has been provided by several
authors,
e.g.,~\cite{bgl_narcowich02_1,bgl_narcowich02_2,bgl_levesley05} for
the sphere and~\cite{bgl_yoon01,bgl_brownlee04,bgl_narcowich04} for
the Euclidean case.

\section{The sphere}\label{bgl_sphere}

Let $\sph^d = \{x \in \R^{d+1}:\ \abs{x}=1\}$. Then the geodesic
distance between points $x,y \in \sph^d$ is $d(x,y)=\cos^{-1} xy$,
where $xy$ denotes the usual inner product of vectors in $\R^{d+1}$.
We let $\nu$ denote the normalised rotationally invariant measure on
the sphere and define the inner product
\begin{equation*}
    \ip{f}{g}_{L_2(\sph^d)} := \int_{\sph^{d}}
f(x)g(x)\,\mathrm{d}\nu(x).
\end{equation*}
Let $\norm{\cdotc}_{L_2(\sph^d)} :=
\ip{\cdotc}{\cdotc}_{L_2(\sph^d)}^{1/2}$ and let $L_2(\sph^d)$
denote the set of functions for which $\norm{\cdotc}_{L_2(\sph^d)} <
\infty$. Let $P_n$ be the polynomials of degree $n$ in $\R^{d+1}$
restricted to the sphere, and let $H_n = P_n \cap P_n^\perp$ be the
space of degree $n$ spherical harmonics. Then, $L_2(\sph^d)$ has the
decomposition
\begin{equation*}
    L_2(\sph^d) = \bigoplus_{n \geq 0} H_n.
\end{equation*}
Let $Y^n_1,\dotsc,Y^n_{d_n}$ be an orthonormal basis for $H_n$.

Related to $\sph^d$ (we will see why shortly), we have the
Gegenbauer polynomials  $C_n^{(\lambda)}(t)$ which are orthogonal on
$[-1,1]$ with respect to the weight $(1-t^2)^{\lambda-1/2}$. It is
well known (M\"uller~\cite{bgl_muller66}, for instance) that the
following addition formula holds:
\begin{equation*}
C_n^{(\lambda)}(xy) = \sum_{j=1}^{d_n} Y^n_j(x) Y^n_j(y),
\end{equation*}
with $\lambda=d/2-1$. The normalisation of the Gegenbauer
polynomials is chosen so that there is no constant in the addition
theorem. It is straightforward to see that $C_n^{(\lambda)}$ is the
kernel of $T_n$, the orthogonal projector from $L_2(\sph^d)$ onto
$H_n$. Thus,
\begin{equation*}
    (T_n f)(x) = \int_{\sph^d} f(y) C_n^{(\lambda)}(xy)\,\mathrm{d}\nu(y),\qquad \text{
    for all $f \in L_2(\sph^d)$.}
\end{equation*}

The following lemma is a specialisation of a result in
\cite{bgl_levesley03} to the sphere.
\begin{lemma}\label{bgl_normsrat}
For $n\geq 0$,
\begin{equation*}
 \norm{T_n f}_{L_\infty(\sph^d)} \leq \sqrt{d_n}\norm{T_n
 f}_{L_2(\sph^d)},\qquad\text{for all $f \in L_2(\sph^d)$.}
\end{equation*}
\end{lemma}

We will be considering interpolation using kernels of the form
$\phi(d(x,y))$ where $\phi:[0,\pi] \rightarrow \R$. We will assume
that the function $\phi$ has an expansion \begin{equation*}
    \phi(d(x,y)) = \sum_{n \geq 0} a_n C_n^{(\lambda)}(xy),
\end{equation*}
where $a_n>0$, for $n=0,1,\dotsc$, and
\begin{equation*}
    \sum_{n \geq 0} d_n a_n < \infty.
\end{equation*}
The first condition ensures that $\phi$ is positive definite, and
the second that it is continuous. Our analysis will take place in
the native space for $\phi$, $\hilb_\phi$, defined by
\begin{equation*}
    \hilb_\phi := \biggl\{f \in L_2(\sph^d):\
    \norm{f}_\phi := \biggl( \sum_{n \geq 0} a_n^{-1}
    \norm{T_n f}_{L_2(\sph^d)}^2 \biggr)^{\frac{1}{2}} < \infty \biggr\}.
\end{equation*}

A pseudodifferential operator $\Lambda$ on $\sph^d$ is an operator
which acts via multiplication by a constant on each eigenspace
$H_{n}$:
\begin{equation*}
    \Lambda p_{n} = \lambda_{n} p_{n},\qquad \text{$p_{n} \in H_{n}$, $n=0,1,\dotsc$}.
\end{equation*}
For more information on pseudodifferential operators on spheres see,
e.g.,~\cite{bgl_freeden98,bgl_svensson83}. We call the sequence of
numbers $\{\lambda_{n}\}_{n \geq 0}$ the \textit{symbol} of
$\Lambda$. Let $\delta_x$ denote the point evaluation functional at
$x$, and, when it makes sense for the functional $\mu$, let $\Lambda
\mu(f) = \mu(\Lambda f)$. Let us denote by $Y^*$ the span of the
point evaluation functionals supported on $Y$. In Morton and
Neamtu~\cite{bgl_morton02} the authors give error estimates for the
collocation solution of pseudodifferential equations on spheres.
Here we attempt, initially, to find errors in pseudoderivatives of
solutions to the interpolation problem.

\begin{proposition} \label{bgl_funcbnd}
Let $S_\phi^Y$ be the $\phi$\nobreakdash-spline interpolant to $f
\in \hilb_\phi$ on the point set $Y \subset \sph^d$. Let $\Lambda$
be a pseudodifferential operator. Then, for each $x \in \sph^d$,
\begin{equation*}
    \abs{\Lambda (f-S_\phi^Y)(x)} \le \inf_{\mu \in Y^*}
    \sup_{\substack{v \in \hilb_\phi\\ \norm{v}_\phi=1}} \abs{\Lambda v
    (x) -  \mu(v)}  \norm{f-S_\phi^Y}_\phi .
\end{equation*}
\end{proposition}
\begin{proof}
Since $f(y)-S_\phi(y)=0$, $y \in Y$, we have, for any coefficients
$c_y$, $y \in Y$,
\begin{align*}
    \abs{\Lambda (f-S_\phi^Y)(x)} &= \biggl| \Lambda (f-S_\phi^Y)(x) - \sum_{y \in Y} c_y (f(y)-S_\phi^Y(y)) \biggr| \\
    &=  \biggl| \Lambda (f-S_\phi^Y)(x) - \sum_{y \in Y} c_y (f(y)-S_\phi^Y(y)) \biggr| \frac{\norm{ f-S_\phi^Y}_\phi}{\norm{ f-S_\phi^Y}_\phi} \\
    &\le \sup_{\substack{v \in \hilb_\phi\\ \norm{v}_\phi=1}}  \biggl| \Lambda v(x) - \sum_{y \in Y} c_y
    v(y) \biggr| \norm{ f-S_\phi^Y}_\phi.
\end{align*}
We now take the infimum over all functionals in $Y^*$ to obtain the
result.
\end{proof}

In what follows we will need the pseudodifferential operator
$\Lambda$ to satisfy the following assumption:
\begin{assumption}\label{bgl_ass}
For all $n \ge 0$, $\lambda_n=(n(d+n-2))^s$, for some $s>0$.
\end{assumption}
From Ditzian \cite{bgl_ditzian98}, if $\Lambda$ satisfies
Assumption~\ref{bgl_ass},  then for $p \in P_n$,
\begin{equation*}
 \norm{\Lambda p}_{L_\infty(\sph^d)} \le  E \lambda_n \norm{p}_{L_\infty(\sph^d)},
\end{equation*}
for some $E$ independent of $n$.

From \cite[Lemma 7]{bgl_jetter99} we have the following result.
\begin{lemma}
Let $Y$ be a finite set of points with fill-distance $h(Y,\sph^d)
\le 1/(2N)$, for some fixed $N \in \Z_+$. Then, for any linear
functional $\gamma$ on $P_N$ with
\begin{equation*}
\sup_{\substack{p \in P_N\\ \norm{p}_{L_\infty(\sph^d)}=1}}
\abs{\gamma p} \leq 1,
\end{equation*}
there is a set of real numbers $\{ b_y \}_{y \in Y}$, with $\sum_{y
\in Y} \abs{b_y} \leq 2$, such that
\begin{equation*}
\gamma p = \sum_{y \in Y} b_y p(y), \qquad \text{for all $p \in
P_N$}.
\end{equation*}
\end{lemma}

Now, for a fixed $x \in \sph^d$, let
\begin{equation*}
\gamma p = \frac{\Lambda p (x)}{E \lambda_N},\qquad \text{for all $p
\in P_N$.}
\end{equation*}
Then,
\begin{equation*}
\sup_{0\neq p \in P_N} \frac{ \abs{\gamma p}}{\norm{ p
}_{L_\infty(\sph^d)} } \le 1,
\end{equation*}
so that, by the previous lemma, there is a set of coefficients
$\{b_y\}_{y \in Y}$, such that
\begin{equation*}
\gamma p = \sum_{y \in Y} b_y p(y),
\end{equation*}
with $\sum_{y \in Y} \abs{b_y} \le 2$. Thus, with $c_y=E\lambda_N
b_y$, for $y \in Y$, we have
\begin{equation}
 \Lambda p (x)  = \sum_{y \in Y} c_y
p(y),\qquad \text{for all $p \in P_N$,}\label{bgl_polyrep}
\end{equation}where,
\begin{equation}
\sum_{y \in Y} \abs{c_y} \le 2 E \lambda_N. \label{bgl_coeffbnd}
\end{equation}

We now arrive at the first main result of this section.
\begin{theorem} \label{bgl_derivcon}
Let $S_\phi^Y$ be the $\phi$\nobreakdash-spline interpolant to $f
\in \hilb_\phi$, on the point set $Y \subset \sph^d$, where
$h(Y,\sph^d) \le 1/(2N)$, for some fixed $N \in \Z_+$. Let $\Lambda$
be a pseudodifferential operator with symbol  $\{\lambda_{n} \}_{n
\geq 0}$ satisfying Assumption \ref{bgl_ass} and
\begin{equation*}
\sum_{n \geq 0} d_{n} \lambda_{n}^2 a_{n} < \infty.
\end{equation*}
Then, for $x \in \sph^d$,
\begin{equation*}
  \abs{\Lambda (f-S_\phi^Y)(x)} \le (1+2E) \biggl( \sum_{n \ge N}
d_n \lambda_n^2 a_n \biggr)^{\frac{1}{2}} \norm{f-S_\phi^Y}_\phi .
\end{equation*}
\end{theorem}
\begin{proof} Let is choose $\{c_y\}_{y \in Y}$ to be the coefficients described above.
Let $v \in \hilb_\phi$ with $\norm{v}_\phi = 1$. Then,
\begin{align*}
\inf_{\mu \in Y^*} \abs{ \Lambda v(x)  - \mu(v) } &\le \biggl| \sum_{n \geq 0} \biggl( \Lambda T_n v(x) - \sum_{y \in Y} c_y T_n v(y) \biggr) \biggr|\\
&= \biggl| \sum_{n>N} \biggl( \lambda_n T_nv (x) - \sum_{y \in Y}
c_y T_n v(y) \biggr) \biggr|,
\end{align*}
by (\ref{bgl_polyrep}). Thus,
\begin{align*}
\inf_{\mu \in Y^*} \abs{\Lambda v(x)  - \mu(v)} &\leq  \biggl| \sum_{n>N} \lambda_n T_n v(x) \biggr| +  \biggl| \sum_{n>N}  \sum_{y \in Y} c_y T_n v(y) \biggr| \\
&\le \sum_{n > N} \biggl( \lambda_n + \sum_{y \in Y} \abs{c_y} \biggr) \norm{T_n v}_{L_\infty(\sph^d)} \\
&\le  \sum_{n > N} \biggl( \lambda_n + \sum_{y \in Y} \abs{c_y}
\biggr) \sqrt{d_n} \norm{T_n v}_{L_2(\sph^d)},
\end{align*}
using Lemma~\ref{bgl_normsrat}. Hence, using (\ref{bgl_coeffbnd})
and the Cauchy--Schwarz inequality,
\begin{equation*}
\inf_{\mu \in Y^*} \abs{ \Lambda v(x)  - \mu(v) } \le \biggl[
\biggl( \sum_{n > N} d_n \lambda_n^2 a_n \biggr)^{\frac{1}{2}} + 2 E
\lambda_N \biggl( \sum_{n > N} d_n a_n \biggr)^{\frac{1}{2}} \biggr]
\norm{v}_\phi,
\end{equation*}
and the result follows from Proposition~\ref{bgl_funcbnd} since
$\norm{v}_\phi = 1$, and because $\{\lambda_n\}_{n\geq0}$ is an
increasing sequence.
\end{proof}

Integrating the conclusion of the previous theorem over the sphere
we easily obtain
\begin{corollary}\label{bgl_l2pseud}
Under the hypotheses of Theorem~\ref{bgl_derivcon},
\begin{equation*}
    \norm{\Lambda (f-S_\phi^Y)}_{L_2(\sph^d)} \leq (1+2E) \biggl(
    \sum_{n > N} d_n \lambda_n^2 a_n \biggr)^{\frac{1}{2}}
    \norm{f-S_\phi^Y}_\phi.
\end{equation*}
\end{corollary}

Before we give our improved error estimate we need to define a new
space $\hilb_{\Lambda \phi}$ by
\begin{equation*}
    \hilb_{\Lambda \phi} := \biggl\{ f \in \hilb_\phi:\
    \norm{f} _{\Lambda \phi} := \biggl( \sum_{n \geq 0}
    (\lambda_n a_n)^{-2} \norm{T_n f}_{L_2(\sph^d)}^2 \biggr)^{\frac{1}{2}}
    < \infty \biggr\}.
\end{equation*}
\begin{theorem}
Let $S_\phi^Y$ be the $\phi$\nobreakdash-spline interpolant to $f
\in \hilb_\phi$ on the point set $Y \subset \sph^d$, where
$h(Y,\sph^d) \leq 1/(2N)$, for some fixed $N \in \Z_+$. Let
$\Lambda$ be a pseudodifferential operator with symbol
$\{\lambda_{n}\}_{n\geq 0}$ satisfying Assumption \ref{bgl_ass} and
\begin{equation*}
    \sum_{n \geq 0} d_{n} \lambda_{n}^2 a_{n} < \infty.
\end{equation*} Then, for $f \in \hilb_{\Lambda \phi}$ and for all $x \in \sph^d$,
\begin{equation*}
    \abs{f(x)-S_\phi^Y(x)} \leq (1+2E) \biggl(\sum_{n > N}d_n
    \lambda_n^2 a_n \biggr)^{\frac{1}{2}} \biggl( \sum_{n > N} d_n a_n
    \biggr)^{\frac{1}{2}} \norm{f}_{\Lambda \phi}.
\end{equation*}
\end{theorem}
\begin{proof}
Firstly, using Proposition~\ref{bgl_normmin} and the Cauchy--Schwarz
inequality, we have
\begin{align*}
    \norm{f-S_\phi^Y}_\phi^2
    &= \ip{f-S_\phi^Y}{f}_\phi \\
    &= \sum_{n \geq 0} a_n^{-1} \ip{T_n(f - S_\phi^Y)}{T_n f}_{L_2(\sph^d)} \\
    &\leq  \biggl( \sum_{n \geq 0}
    \lambda_n^2 \norm{ T_n (f-S_\phi^Y) }_{L_2(\sph^d)}^2\biggr)^{\frac{1}{2}}
    \biggl( \sum_{n \geq 0} (\lambda_n a_n)^{-2}
    \norm{T_n f}_{L_2(\sph^d)}^2 \biggr)^{\frac{1}{2}}\\
    &= \norm{ \Lambda (f-S_\phi^Y) }_{L_2(\sph^d)} \norm{f}_{\Lambda \phi} \\
    &\leq  (1+2E) \biggl( \sum_{n > N} d_n \lambda_n^2 a_n \biggr)^{\frac{1}{2}}
    \norm{f-S_\phi^Y}_\phi \norm{f}_{\Lambda \phi},
\end{align*}
using Corollary~\ref{bgl_l2pseud}. Cancelling a factor of
$\norm{f-S_\phi^Y}_\phi$ from both sides yields
\begin{equation*}
    \norm{f-S_\phi^Y}_\phi \leq (1+2E) \biggl( \sum_{n >
    N} d_n \lambda_n^2 a_n \biggr)^{\frac{1}{2}}  \norm{f}_{\Lambda \phi}.
\end{equation*}
We can now employ the standard error estimate taken from Jetter,
St\"ockler and Ward~\cite{bgl_jetter99} (our
Theorem~\ref{bgl_derivcon} with $\lambda_n=1$ for all $n$) to give
the required result.
\end{proof}

\section{The Euclidean case} \label{bgl_euclid}

Our attention now turns to $\phi$\nobreakdash-spline interpolants of
the form
\begin{equation*}
    S_\phi^Y(x)= \sum_{y \in Y} \alpha_y \phi(\abs{x-y}),
\end{equation*}
where $\phi:\R_+ \rightarrow \R$. We will conduct our analysis for
positive definite basis functions $\phi \in L_1(\R^d)$ whose Fourier
transform satisfy, for some $s >0$,
\begin{equation}\label{bgl_decay}
    C_1 (1+\abs{x})^{-2s} \leq \widehat{\phi}(x) \leq C_2
    (1+\abs{x})^{-2s},
\end{equation}
for some positive constants $C_1$ and $C_2$, for example, the
Sobolev splines~\cite{bgl_dix94} or piecewise polynomial compactly
supported radial functions of minimal degree~\cite{bgl_wendland95}.
The expo\-sition contained in this section can be readily adapted to
include the polyharmonic splines~\cite{bgl_duchon78} as well. In
that case, the $\phi$\nobreakdash-spline interpolant must be
augmented by a polynomial $p$ with the extra degrees of freedom
taken up by the side conditions
\begin{equation*}
  \sum_{y \in Y} \alpha_y q(y) =0,
\end{equation*}
where $q$ is polynomial of the same degree (or less) as $p$.

For a domain $\Omega \subset \R^d$ let $L_2(\Omega)$ denote the
usual space of square-integrable functions on $\Omega$ with inner
product $\ip{\cdotc}{\cdotc}_{L_2(\Omega)}$ and norm
$\norm{\cdotc}_{L_2(\Omega)}$. For $k \in \Z_+$, the integer-order
Sobolev space is defined by
\begin{equation*}
    \hilb_{k} := \biggl\{f \in L_2(\R^d):\ D^\alpha f \in
    L_2(\R^d)\ \, \text{for all $\abs{\alpha}\leq k$}\biggr\},
\end{equation*}
with $D^\alpha$ understood in the distributional sense, which
carries the inner product
\begin{equation*}
    \ip{f}{g}_k := \ip{f}{g}_{L_2(\R^d)} + (f,g)_k,
\end{equation*}
where $(f,g)_k$ denotes the Sobolev semi-inner product
\begin{equation*}
    (f,g)_{k} :=  \sum_{\abs{\alpha} = k} c_\alpha^{(k)}
    \int_{\R^d} (D^\alpha f)(x) (\overline{D^\alpha g})(x)
    \,\mathrm{d}x,
\end{equation*}
with associated semi-norm $\abs{\cdotc}_k :=
(\cdotc,\cdotc)^{1/2}_k$. The coefficients $c_\alpha^{(k)}$ have
been chosen so that
\begin{equation*}
    \sum_{\abs{\alpha}=k} c_\alpha^{(k)} x^{2\alpha} = \abs{x}^{2k}.
\end{equation*}

We can write the semi-norm, using the Fourier transform, in the
alternative form
\begin{equation*}
    \abs{f}_{k}^2 = \int_{\R^d} \abs{\widehat{f}(x)}^2
    \abs{x}^{2k} \,\mathrm{d}x,
\end{equation*}
which facilitates the definition of fractional-order Sobolev space,
$\hilb_{s}$, for $s >0$, which has the semi-norm
\begin{equation}\label{bgl_altform}
    \abs{f}_{s}^2 :=  \int_{\R^d} \abs{\widehat{f}(x)}^2
    \abs{x}^{2s} \,\mathrm{d}x.
\end{equation}
The space $\hilb_{s}$ is complete with respect to
\begin{equation*}
  \norm{f}_{s} := \left\{
  \begin{aligned}
        &\Bigl( \norm{f}_{L_2(\R^d)}^2+\abs{f}_{s}^2\Bigr)^{\frac{1}{2}}&\quad& \text{if $s \in \Z_+$,} \\
        &\Bigl( \norm{f}_{\lfloor s \rfloor}^2+\abs{f}_{s}^2\Bigr)^{\frac{1}{2}}&\quad& \text{otherwise,}
    \end{aligned}\right.
\end{equation*}
and, whenever we have $s>d/2$, $\hilb_{s}$ is continuously embedded
in the continuous functions. The native space for $\phi$
satisfying~\eqref{bgl_decay} is equivalent to $\hilb_s$.

We now wish to make local definitions of our function spaces, which
we shall denote by $\hilb_{s}(\Omega)$. For $s \in \Z_+$ the
definition should be is clear. In what follows we also need the
local fractional-order Sobolev spaces:
\begin{equation*}
  \hilb_{s}(\Omega) := \biggr\{f \in \hilb_{\lfloor s \rfloor}(\Omega):\ \norm{f}_{s,\Omega} :=
  \Bigl( \norm{f}_{\lfloor s \rfloor,\Omega}^2+\abs{f}_{s,\Omega}^2\Bigr)^{\frac{1}{2}}<\infty\biggl\},
\end{equation*}
where $\abs{f}_{s,\Omega}$ is the local fractional-order Sobolev
semi-norm obtained by rewriting~\eqref{bgl_altform} in an equivalent
direct form, i.e., not defined through the Fourier transform of $f$
(see, e.g., Adams~\cite[p. 214]{bgl_adams75}). For our analysis we
find it more useful to exploit an equivalent wavelet representation
for the local Sobolev norm~\cite{bgl_cohen00}.

To introduce this equivalent norm we stipulate that the bounded
domain, $\Omega$, admits a local multiresolution of closed subspaces
$\{V_n(\Omega)\}_{n\geq 0}$ of $L_2(\Omega)$:
\begin{equation*}
    V_0(\Omega) \subset V_1(\Omega) \subset \dotsb \subset
    L_2(\Omega),\qquad \overline{\bigcup_{n \geq 0} V_n(\Omega)} =
    L_2(\Omega).
\end{equation*}
Cohen et al.~\cite{bgl_cohen00} give sufficient conditions on
$\Omega$ to admit such a local multiresolution. In particular, for
$d=2$, those domains whose boundaries have certain piecewise
Lipschitz smoothness are admissible. The following is an incidence
of~\cite[Theorem 4.2]{bgl_cohen00}.
\begin{theorem}\label{bgl_wavelet}
  Suppose $\Omega \subset \R^d$ is a bounded domain that
  admits a local multi\-resolution $\{V_n(\Omega)\}_{n\geq 0}$ for
  $L_2(\Omega)$. For $n \geq 0$, let $Q_n^\Omega$
  denote the orthogonal projection from $L_2(\Omega)$ onto
  $W_n(\Omega) = V_{n}(\Omega) \ominus V_{n-1}(\Omega)$ with the convention
  that $V_{-1}(\Omega) = \{0\}$. For each $s \ge 0$, let $\Lambda_s$
  be the pseudodifferential operator on $\Omega$ defined via
  \begin{equation*}
    \Lambda_s := \sum_{n\geq 0} 2^{ns} Q_n^\Omega.
  \end{equation*}
  Then, there exists positive
  constants $C_1$ and $C_2$ such that, for all $f \in \hilb_{s}(\Omega)$,
  \begin{equation*}
    C_1 \norm{\Lambda_s f}_{L_2(\Omega)} \leq \norm{f}_{s,\Omega} \leq
    C_2 \norm{\Lambda_s f}_{L_2(\Omega)}.
  \end{equation*}
\end{theorem}

Now, let us return to the task at hand. For $\phi$
satisfying~\eqref{bgl_decay}, we will denote the
$\phi$\nobreakdash-spline interpolant on the point set $Y$ by
$S_\phi^Y$. The standard error estimate in this context is
\begin{equation} \label{bgl_standard}
\abs{f(x) - S_\phi^Y(x)} \le C h^{s-d/2} \norm{f - S_\phi^Y}_{s},
\end{equation}
see~\cite{bgl_schaback93}. If $f$ is smoother (and satisfies some
boundary conditions) we can get a better rate of convergence.

We will exploit the fact that $\hilb_{\mu}(\Omega)$, for $0<\mu<s$,
is an interpolation space lying between $L_2(\Omega)$ and
$\hilb_{s}(\Omega)$ (see Bergh and L{\"o}fstr{\"o}m~\cite[p.
131]{bgl_bergh76}). We can then use the standard interpolation
theorem concerning the norms of operators bounded on the extreme
spaces to infer a bound on the norm for the interpolation space. For
further information on interpolation spaces the reader can consult,
e.g., Bergh and L{\"o}fstr{\"o}m~\cite{bgl_bergh76}. We use the
following interpolation theorem.
\begin{proposition}
Let $0 < \mu < s$. Further suppose that $T:\hilb_{s}(\Omega)
\rightarrow L_2(\Omega)$, and $T:\hilb_{s}(\Omega) \rightarrow
\hilb_{s}(\Omega)$ is a bounded operator. Then,
\begin{equation*}
\norm{T}_{\hilb_{s}(\Omega) \rightarrow \hilb_{\mu}(\Omega)} \le
\norm{T}^{1-\mu/s}_{\hilb_{s}(\Omega) \rightarrow L_2(\Omega)}
\norm{T}^{\mu/s}_{\hilb_{s}(\Omega) \rightarrow \hilb_{s}(\Omega)}.
\end{equation*}
\end{proposition}

Since we can write the $\hilb_s$-norm in entirely direct form, we
are at liberty to utilise Duchon's localisation
technique~\cite{bgl_duchon78} to enhance the standard error
estimate~\eqref{bgl_standard}. Therefore, if $\Omega$ is bounded and
satisfies an interior cone condition then, for $f \in
\hilb_{s}(\Omega)$, $s>d/2$, and sufficiently small $h=h(Y,\Omega)$,
\begin{equation*}
\norm{f-S_\phi^Y}_{L_2(\Omega)} \le C h^s
\norm{f-S_\phi^Y}_{s,\Omega}.
\end{equation*}
Writing $T f = f-S_\phi^Y$, we see, using the last proposition, that,
for $0 < \mu < s$,
\begin{align}
\norm{f-S_\phi^Y}_{\mu,\Omega} &\le (Ch^s
\norm{f-S_\phi^Y}_{s,\Omega})^{1-\mu/s}
\norm{f-S_\phi^Y}_{s,\Omega}^{\mu/s}  \nonumber\\
&= C h^{s-\mu} \norm{f-S_\phi^Y}_{s,\Omega}. \label{bgl_interp}
\end{align}

We can now prove our main result of this section, which is a
generalisation of that of Schaback~\cite{bgl_schaback99}.
\begin{theorem}\label{bgl_euclid_main}
Suppose $\Omega \subset \R^d$ is bounded, satisfies an interior cone
condition and admits a local multiresolution. Let $s>d/2$ and let
$S_\phi^Y$ be the $\phi$\nobreakdash-spline interpolant to $f \in
\hilb_{s}$ on the point set $Y \subset \Omega$. Suppose further that
$f \in \hilb_{\nu}$, for $s < \nu \leq 2s$, and that $f$ is
compactly supported in $\Omega$. Then there exists $C>0$,
independent of $f$ and $h=h(Y,\Omega)$, such that for all $x \in
\Omega$ and sufficiently small $h$,
\begin{equation*}
\abs{f(x) - S_\phi^Y(x)} \le Ch^{\nu - d/2} \norm{f}_{\nu,\Omega}.
\end{equation*}
\end{theorem}
\begin{proof}
From Proposition~\ref{bgl_normmin} we know that
\begin{equation*}
    \ip{f-S_\phi^Y}{S_\phi^Y}_{s} = 0,
\end{equation*}
so that
\begin{equation*}
    \norm{f-S_\phi^Y}_{s}^2 = \ip{f-S_\phi^Y}{f}_{s} \leq C
    \ip{f-S_\phi^Y}{f}_{s,\Omega},
\end{equation*}
where we have used the compact support of $f$ in $\Omega$. Now, the
equivalent norm from Theorem~\ref{bgl_wavelet} gives us
\begin{align*}
    \norm{f-S_\phi^Y}_{s}^2 &\leq C \ip{\Lambda_s(f-S_\phi^Y)}{\Lambda_s
    f}_{L_2(\Omega)}\\
    &= C \sum_{n\geq 0} 4^{ns}
    \ip{Q^\Omega_n (f-S_\phi^Y)}{Q^\Omega_n f}_{L_2(\Omega)},
\end{align*}
and successive applications of the continuous and discrete
Cauchy--Schwarz inequality yields
\begin{align*}
    \norm{f-S_\phi^Y}_{s}^2 &\leq C \sum_{n\geq 0} \norm{2^{n(2s-\nu)} Q^\Omega_n
    (f-S_\phi^Y)}_{L_2(\Omega)}
    \norm{2^{n\nu} Q^\Omega_n f}_{L_2(\Omega)}\\
    &\leq C \biggl(\sum_{n\geq 0} \norm{2^{n(2s-\nu)} Q^\Omega_n
    (f-S_\phi^Y)}_{L_2(\Omega)}^2\biggr)^{\frac{1}{2}}\biggl(\sum_{n\geq 0}
    \norm{2^{n\nu} Q^\Omega_n
    f}_{L_2(\Omega)}^2\biggr)^{\frac{1}{2}}\\
    &= C \norm{\Lambda_{2s-\nu} (f-S_\phi^Y)}_{L_2(\Omega)} \norm{\Lambda_\nu
    f}_{L_2(\Omega)}.
\end{align*}
Thus, using the norm equivalence from Theorem~\ref{bgl_wavelet}
again together with~\eqref{bgl_interp}, we have
\begin{align*}
    \norm{f-S_\phi^Y}_{s}^2 &\leq C \norm{f-S_\phi^Y}_{2s-\nu,\Omega}
    \norm{f}_{\nu,\Omega}\\
    &\leq C h^{\nu-s}
    \norm{f-S_\phi^Y}_{s,\Omega}\norm{f}_{\nu,\Omega}\\
    &\leq C h^{\nu-s}
    \norm{f-S_\phi^Y}_{s}\norm{f}_{\nu,\Omega},
\end{align*}
and cancelling one power of $\norm{f-S_\phi^Y}_{s}$ gives
\begin{equation*}
\norm{f-S_\phi^Y}_{s} \le C h^{\nu-s} \norm{f}_{\nu,\Omega}.
\end{equation*}
The result follows by substitution into the standard error estimate
\eqref{bgl_standard}.
\end{proof}

\section*{Acknowledgements} We are grateful to Professor Mikhail
Shubin for useful input into questions of pseudodifferential
operators and alternative Sobolev space semi-norms.

\bibliographystyle{siam}

\end{document}